\documentclass[12pt]{article}
\usepackage{pdfsync}
\usepackage{graphicx}
\usepackage{amssymb,amsmath,latexsym,amsfonts,amsthm}
\usepackage{epstopdf}
\usepackage[mathscr]{eucal}
\DeclareMathOperator{\de}{d}

\DeclareMathOperator{\dom}{dom}
\DeclareMathOperator{\Sp}{Sp}

\DeclareMathOperator{\dist}{dist}
\DeclareMathOperator{\tset}{tset}

\begin{document}

\title{Distance sets of universal and Urysohn metric spaces}
\author{ N. W. Sauer}

\maketitle

\newcommand{\snl} {\\ \smallskip}
\newcommand{\mnl}{\\ \medskip} 
\newcommand{\Bnl} {\\ Bigskip}
\newcommand{\edge}{\makebox[22pt]{$\circ\mspace{-6 mu}-\mspace{-6 mu}\circ$}}
\newcommand{\nedge}{\makebox[22pt]{$\circ\mspace{-6 mu}\quad\: \mspace{-6 mu}\circ$}}

\newcommand{\restrict}[2]{#1\mspace{-2mu}\mathbin{\upharpoonright}\mspace{-1mu} #2}
\newcommand{\Kat}{Kat\u{e}tov }
\newcommand{\Fra}{Fra\"{\i}ss\'e  }
\newcommand{\str}[1]{\stackrel{#1}{\sim}}
\newcommand{\spe}{\mathrm{spec}}
\newcommand{\UR}{\boldsymbol{U}_{\hskip -3pt R}}

\newcommand{\upl}[1]{\sideset{ }{^{c}}{\operatorname{\mathit{{#1}}}}}
\newcommand{\lpl}[1]{\sideset{ }{_{\negmedspace c}}{\operatorname{\mathit{{#1}}}}}
\newcommand{\ipl}[1]{\sideset{ }{_{\negmedspace c}^{c}}{\operatorname{\mathit{{#1}}}}}

\newcommand{\last}[1]{\mathfrak{#1}\!\rceil}

\newtheorem{thm}{Theorem}[section]
\newtheorem{thmb}{Theorem}
\newtheorem{lem}{Lemma}[section]
\newtheorem{coroll}{Corollary}[section]
\newtheorem{ass}{Assumption}[section]
\newtheorem{defin}{Definition}[section]
\newtheorem{example}{Example}[section] 
\newtheorem{fact}{Fact}[section]

\newtheorem{prop}{Proposition}[section]
\newtheorem{obs}{Observation}[section]
\newtheorem{cor}{Corollary}[section]
\newtheorem{sublem}{Sublemma}[section]
\newtheorem{claim}{Claim}
\newtheorem{question}{Question}[section]

\newtheorem{comment}{Comment}[section]
\newtheorem{problem}{Problem}[section]
\newtheorem{remark}{Remark}[section]

\newcommand\con{\char'136{}}

\begin{abstract}

A metric space $\mathrm{M}=(M;\de)$ is  {\em homogeneous} if for every isometry $f$ of a finite subspace of $\mathrm{M}$ to a subspace of $\mathrm{M}$ there exists an isometry of $\mathrm{M}$ onto $\mathrm{M}$ extending $f$. A metric space $\boldsymbol{U}$   is an {\em Urysohn} metric space if it is homogeneous and separable and complete and if it isometrically embeds every separable metric space $\mathrm{M}$ with $\dist(\mathrm{M})\subseteq \dist(\boldsymbol{U})$.  (With $\dist(\mathrm{M})$ being the set of distances between points in $\mathrm{M}$.)

The main results are:  (1) A characterization of the sets $\dist(\boldsymbol{U})$ for Urysohn metric spaces $\boldsymbol{U}$.  (2) If $R$ is the distance set of an Urysohn metric space and $\mathrm{M}$ and $\mathrm{N}$ are two metric spaces, of any cardinality with distances in $R$, then they amalgamate disjointly to a metric space with distances in $R$.   (3) The completion of a homogeneous  separable  metric space $\mathrm{M}$ which embeds isometrically  every finite metric space $\mathrm{F}$ with $\dist(\mathrm{F})\subseteq \dist(\mathrm{M})$   is homogeneous. 
\end{abstract}

\section{Introduction}

There exists a homogeneous separable complete metric space, the classical {\em Urysohn metric space} $\boldsymbol{U}_{\hskip -2pt\Re_{\geq 0}}$, which embeds every separable metric space, see \cite{Ury}.  In this paper we will discuss the question: For which subsets $R$ of the reals $\Re$ and for which subsets  of the set of properties of being homogeneous, separable, complete, embedding every finite metric space with distances in $R$ or even embedding every separable metric space with distances in $R$, does there exist a metric space $\mathrm{M}$ with $\dist(\mathrm{M})= R$? It is for example well known that there exists a unique homogeneous, separable, complete metric space which isometrically embeds every separable metric space with set of distances a subset of the interval $[0,1]$. This space is the {\em Urysohn sphere} $\boldsymbol{U}_{\hskip -2pt[0,1]}$.

In \cite{KPT}, Kechris, Pestov and Todorcevic established a connection between structural Ramsey theory and automorphism groups of homogeneous relational structures and there  and in  particular  also in  \cite{Pe1}   the notion of oscillation stability of such groups is defined. It is  shown that this notion is equivalent to a partition problem in the case of homogeneous metric spaces. A metric space $\mathrm{M}=(M;\de)$  being {\em oscillation stable} if for $\epsilon>0$  and $f:M\to \Re$ bounded and uniformly continuous there exists a copy $\mathrm{M}^\ast=(M^\ast;\de)$ of $\mathrm{M}$ in $\mathrm{M}$ so that:
\[
\sup\{|f(x)-f(y)|\mid x,y\in M^\ast\}<\epsilon.
\]
Prompted by results of V. Milman \cite{Mil} it is shown in \cite{OS} that the Hilbert sphere in $\ell_2$ is not oscillation stable and in \cite{LNSA} that the Urysohn sphere $\boldsymbol{U}_{\hskip -1pt[0,1]}$  is oscillation stable. In particular it was shown in \cite{LNSA} that the Urysohn metric spaces $\boldsymbol{U}_{\hskip -2pt n}$ are indivisible, which due to the main  result in \cite{LANVT} implies the oscillation stability of $\boldsymbol{U}_{\hskip -2pt n}$. ($\boldsymbol{U}_{\hskip -2pt n}$  is the Urysohn metric space with $\dist(\boldsymbol{U}_{\hskip -2pt n})=\{0,1,2,\dots,n-1\}$.) It is shown in  \cite{SA1} that all Urysohn metric spaces with a finite distance set are indivisible. In order to determine if Urysohn metric spaces, or which Urysohn metric spaces,  are oscillation stable we characterize in this paper the  distance sets of Urysohn metric spaces; providing a baseline for the set of spaces to be investigated for oscillation stability. (According to Theorem 1 below, any two Urysohn metric spaces with the same set of distances being isometric.)   See  \cite{Ng} for  details and additional references on oscillation stability.

J. Clemens in \cite{Cl} proved that: Given a set of non-negative reals, $R \subseteq  \Re_{\geq 0}$, the set $R$ is the set of distances for some complete and separable metric space if and only if $R$ is an analytic set containing 0 and either$R$ is countable or 0 is a limit point of $A$.  Clemens then asks to determine distance sets of metric spaces which are homogeneous.  The following three definitions describe the metric spaces under consideration and define the basic tool for their characterization:

\begin{defin}\label{defin:uni}
A metric space $\mathrm{M}$ is {\em universal} if it is homogeneous and isometrically embeds every finite metric space $\mathrm{F}$ with $\dist(\mathrm{F})\subseteq \dist(\mathrm{M})$. 
\end{defin}

\begin{defin}\label{defin:Uryunif}
A metric space $\boldsymbol{U}$ is an {\em Urysohn} metric space if it is homogeneous and separable and complete and if it isometrically embeds every separable metric space $\mathrm{M}$ with $\dist(\mathrm{M})\subseteq \dist(\boldsymbol{U})$
\end{defin}
Note that every countable universal metric space with a finite set of distances is an Urysohn metric space. 

\begin{defin}\label{defin:4valcondi}
A set $R\subseteq \Re_{\geq 0}$ satisfies the {\em 4-values condition} if:

\noindent
For all pairs of metric spaces $\mathrm{A}=(\{x,y,z\};\de_\mathrm{A})$ and\/ $\mathrm{B}=(\{u,y,z\};\de_\mathrm{B})$ with $\dist(\mathrm{A})\cup\dist(\mathrm{B})\subseteq R$ and with $\de_\mathrm{A}(y,z)=\de_\mathrm{B}(y,z)$ exists a metric space $\mathrm{C}=(\{x,y,z,u\};\de_\mathrm{C})$  with $\dist(\mathrm{C})\subseteq R$ and so that the subspace of $\mathrm{C}$ induced by $\{x,y,z\}$ is equal to $\mathrm{A}$ and the subspace of $\mathrm{C}$ induced by $\{u,y,z\}$ is equal to $B$.
\end{defin}
The notion  4-values condition was first formulated in \cite{DLPS} and will be discussed and used at length in this paper.

We will prove:

\begin{thmb}\label{thm:hauptb} (See Theorem \ref{thm:haupt} and Theorem \ref{thm:Uryuni}.)
Let $0\in R\subseteq \Re_{\geq 0}$ with $0$ as a limit.  Then there exists an  Urysohn metric space $\UR$ if and only if $R$ is   a closed subset of\/ $\Re$ which satisfies the 4-values condition. 

Let $0\in R\subseteq \Re_{\geq 0}$ which does not have $0$ as a limit.  Then there exists an Urysohn metric space $\UR$ if and only if  $R$ is a countable subset of\/ $\Re$ which  satisfies the 4-values condition.

Any two Urysohn metric spaces having the same set of distances are isometric.
\end{thmb}
It  follows  from Theorem 1.4 in \cite{DLPS}  and is  stated in this paper as Theorem \ref{thm:basic1} that if $0\in R\subseteq \Re_{\geq 0}$ is a countable set of numbers which satisfies the 4-values condition then there exists a unique countable universal metric space $\UR$ with $\dist(\UR)=R$. Note that this space $\UR$ is not an Urysohn metric space if $0$ is a limit of $R$ and $R$ is not closed in $\Re$.

\begin{thmb}\label{thm:4necb}(See Theorem \ref{thm:4nec}.)
The set of distances of a universal metric space satisfies the 4-values condition. 
\end{thmb}

Proposition 10 of \cite{Ng} provides an example of a countable homogeneous metric space whose completion is not homogeneous. But, in the case of universal metric spaces we have the following: 
\begin{thmb}\label{thm:join3b} (See  Theorem \ref{thm:afterth}.)
The completion of a universal  separable  metric space $\mathrm{M}$   is homogeneous.  
\end{thmb}
On the other hand,  according to example \ref{ex:compluniv}, the completion  of a universal separable metric space $\boldsymbol{U}$ need not be universal.  That is, the completion $\mathrm{M}$ might not  embed every finite metric space with distances in $\dist(\mathrm{M})$. The next theorem characterizes the finite metric spaces which have an embedding into $\mathrm{M}$. In particular it follows  from the next Theorem that $\dist(\mathrm{M})$ is the closure of $\dist(\boldsymbol{U})$ in $\Re$, Corollary~\ref{cor:embedwh}. See Example \ref{ex:notclosed} for a metric space for which the distance set of its completion is not closed. 

\begin{thmb}\label{thm:embedwhb} (See theorem \ref{thm:embedwh}.)
Let $0\in R\subseteq \Re_{\geq 0}$ be countable,  satisfy the 4-values condition and have $0$ as a limit.  Let  $\mathrm{M}=(M;\de)$ be  the completion of $\UR$. 

A finite metric space $\mathrm{A}=(A;\de_{\mathrm{A}})$ with $A=\{a_i\mid i\in m\}$ has an isometric embedding into $\mathrm{M}$ if and only if for every $\epsilon>0$ there exists a metric space $\mathrm{B}$ with $B=\{b_i\mid i\in m\}$ and distances in $R$ so that $|\de(a_i,a_j)-\de(b_i,b_j)|<\epsilon$ for all $i,j\in m$. 
\end{thmb}

Extending Urysohn's original result   M. Kat\u{e}tov in \cite{Kat}, using ''\Kat functions``, generalized Urysohn's construction to metric spaces which are ''$\kappa$-homogeneous`` and have weight $\kappa$ for $\kappa$ an inaccessible cardinal number. The distance sets of the such constructed Urysohn type spaces are either $\Re_{\geq 0}$ or the unit interval.  More recently those Urysohn spaces attracted attention because of interesting properties of their isometry group, $\mathrm{Iso}(\mathbf{U})$. For example Uspenskij's  result \cite{Uspens} that the isometry group of the Urysohn space is a universal Polish group and the connection of $\mathrm{Iso}(\boldsymbol{U}_{\hskip-2pt[0,1]})$ to minimal topological groups, \cite{Uspens2}. See also Mbombo and Pestove \cite{MPest} and Melleray \cite{Meller} for further discussion. 

We do not follow Kat\u{e}tov's method but lean instead on  the general \Fra theory, see \cite{Fr},  for our results. \Fra theory being particularly well suited for investigating partition problems of separable metric spaces, our main interest. Nevertheless it turned out to be easy to extend the arguments to obtain a general amalgamation result, Theorem~\ref{thm:alamgam3},  for metric spaces whose distance sets are subsets of a closed set of reals satisfying the 4-values condition. This then implies, by extending the \Fra constructions in an obvious way, (see for example \cite{Jonss} or more recently \cite{Barbina} or many other recent generalizations,)  the existence of Urysohn type metric spaces $\boldsymbol{U}$   which are ''$\kappa$-homogeneous`` and have weight $\kappa$ for $\kappa$ an inaccessible cardinal. The distance sets of those spaces  $\boldsymbol{U}$ are closed subsets of $\Re$ satisfying the 4-values condition. Providing another construction for the ''\Kat type metric spaces`` with sets of distances all of $\Re$ or the unit interval.

\section{Notation and \Fra theory}

For another and more detailed introduction to \Fra theory in the context of metric spaces see \cite{Ng}. The exposition here is complete and self contained but might require some, indeed very limited, familiarity with simple model theoretic constructions. 

A pair $\mathrm{H}=(H,\de)$ is a {\em premetric space} if $\de: H^2\to \Re_{\geq 0}$, the {\em distance function of\/ $\mathrm{H}$},   is a function with $\de(x,y)=0$ if and only if $x=y$ and $\de(x,y)=\de(y,x)$ for all $x,y\in H$.  For $A\subseteq H$ we denote by  $\restrict{\mathrm{H}}{A}$ the {\em substructure of\/ $\mathrm{H}$ generated by  $A$}, that is  the premetric space on $A$ with distance function the restriction of\/ $\de$ to $A^2$. The {\em skeleton} of $\mathrm{H}$ is the set of finite induced subspaces  of $\mathrm{H}$ and the {\em age} of $\mathrm{H}$ is the class of finite premetric spaces isometric to some element of the skeleton of $\mathrm{H}$. Let $\dist(\mathrm{H})=\{\de(x,y)\mid x,y\in H\}$.    

A  function $\mathfrak{t} : F\to \Re_{>0}$ with $F$ a finite subset of $H$  is a  {\em type function} of $\mathrm{H} $.   For $\mathfrak{t} $ a type function  let $\Sp(\mathfrak{t})$ be the premetric space on $F\cup \{\mathfrak{t}\}$ for which: 
\begin{enumerate}
\item $\restrict{\Sp(\mathfrak{t})}{\dom(\mathfrak{t})}=\restrict{\mathrm{H} }{\dom(\mathfrak{t})}$.
\item $\forall x\in F\, \,\big(\de(\mathfrak{t},x)=\mathfrak{t}(x)\big)$.
\end{enumerate}

Note  that $\Sp(\mathfrak{t})$ is a metric space if and only if $\restrict{\mathrm{H}}{\dom(\mathfrak{t})}$ is a metric space and  if   for all $x,y\in \dom(\mathfrak{t})$:
\begin{align}
|\mathfrak{t}(x)-\mathfrak{t}(y)|\leq d(x,y)\leq \mathfrak{t}(x)+\mathfrak{t}(y).
\end{align}

For $\mathfrak{t}$ a type function    let
\[
\tset(\mathfrak{t})=\{y\in H\setminus\dom(\mathfrak{t}) : \forall\, x\in  \dom(\mathfrak{t})\, \bigl(\de(y,x)=\mathfrak{t}(x)\bigr)\},
\]
the {\em typeset} of $\mathfrak{t}$. Every  element $y\in \tset(\mathfrak{t})$ is a {\em realization} of $\mathfrak{t}$ in $\mathrm{H}$.   Let $\dist(\mathfrak{t})=\{\mathfrak{t}(x)\mid x\in \dom(\mathfrak{t})\}$.

\begin{defin}\label{defin:Kat}
Let $\mathrm{M}=(M; \de)$ be a metric space. A type function $\mathfrak{k}$ of\/ $\mathrm{M}$ is {\em metric} if\/ $\Sp(\mathfrak{k})$ is a metric space and it is a \Kat function of\/ $\mathrm{M}$ if\/ $\Sp(\mathfrak{k})$ is an element of the age of $\mathrm{M}$. 

A type function of $\mathrm{M}$ is {\em restricted} if it is metric and if $\dist(\mathfrak{k})\subseteq \dist(\mathrm{M})$.
\end{defin}
Note that a type function $\mathfrak{k}$ of a universal metric space $\mathrm{M}$ is a \Kat function if and only if it is restricted.

\begin{lem}\label{lem:realin1}
If every restricted type function of a metric space $\mathrm{M}=(M;\de)$ has a realization in $\mathrm{M}$ then every countable metric space $\mathrm{N}=(N;\de)$ with $\dist(\mathrm{N})\subseteq \dist(\mathrm{M})$ has an isometric embedding into $\mathrm{M}$. 

If every \Kat function of a metric space $\mathrm{M}=(M;\de)$ has a realization in $\mathrm{M}$ then every countable metric space $\mathrm{N}=(N;\de)$  whose age is a subset of the age of $\mathrm{M}$ has an isometric embedding into $\mathrm{M}$. 
\end{lem}
\begin{proof}
Enumerate $N$ into an $\omega$ sequence $(v_i; i\in \omega)$ and for $n\in \omega$  let $N_n=\{v_i\mid i\in n\}$. If $f_n$ is an isometry of $\restrict{\mathrm{N}}{N_n}$ into $\mathrm{M}$ let $f_{n+1}$ be the extension of $f_n$ to an isometry of $\restrict{\mathrm{N}}{N_{n+1}}$ into $\mathrm{M}$ constructed as follows: Let $\mathfrak{k}$ be the type function of $\mathrm{M}$ with $\dom(\mathfrak{k})=f_n[N_n]$ and $\mathfrak{k}(f_n(x))=\de(x,v_n)$. Then $\mathfrak{k}$ is a restricted type function of $\mathrm{M}$ and hence has a realization, say $a$, in $M$. Let $f_{n+1}(v_n)=a$. 

Then $f=\bigcup_{n\in \omega}f_n$ with $f_0$ the empty function is an isometry of $\mathrm{N}$ into $\mathrm{M}$. 

The proof of the second part of the Lemma is analogues.
\end{proof}

\begin{lem}\label{lem:realin3}
Let $\mathrm{M}$ and $\mathrm{N}$ be two countable metric spaces with $\dist(\mathrm{M})=\dist(\mathrm{N})$ and so that every restricted type function of\/ $\mathrm{M}$ has a realization in $\mathrm{M}$ and every restricted type function of\/ $\mathrm{N}$ has a realization in $\mathrm{N}$. 

Or, let $\mathrm{M}$ and $\mathrm{N}$ be two countable metric spaces with equal ages  and so that every \Kat function of\/ $\mathrm{M}$ has a realization in $\mathrm{M}$ and every \Kat function of\/ $\mathrm{N}$ has a realization in $\mathrm{N}$. 

Then every isometry of a finite subspace of\/ $\mathrm{M}$ into $\mathrm{N}$ has an extension to an isometry of\/ $\mathrm{M}$ onto $\mathrm{N}$. 
\end{lem}
\begin{proof}
Extend the proof of Lemma \ref{lem:realin1} to a back and forth argument by alternating the extension of finite isometries between $\mathrm{M}$ and $\mathrm{N}$. (As in the standard proof that every countable dense and unbounded linear order is order isomorphic to the rationals.) 
\end{proof}

\begin{cor}\label{cor:realin3}
Let $\mathrm{M}$ be a countable  metric space so that every \Kat function of\/ $\mathrm{M}$ has a realization in $\mathrm{M}$. Then $\mathrm{M}$ is homogeneous. If every restricted type function of\/ $\mathrm{M}$ has a realization in $\mathrm{M}$ then $\mathrm{M}$ is universal.
\end{cor}

\begin{lem}\label{lem:realin4}
Every  \Kat function of a homogeneous metric space $\mathrm{M}$ has a realization in $\mathrm{M}$. Every restricted type function of a universal metric space $\mathrm{M}$ has a realization in $\mathrm{M}$.
\end{lem}
\begin{proof}
If $\mathrm{M}$ is homogeneous and $\mathfrak{k}$ is a \Kat  function of $\mathrm{M}$,  there exists an isometry $f$ of $\Sp(\mathfrak{k})$ into $\mathrm{M}$. Let $g$ be the restriction of $f$ to $\dom(\mathfrak{k})$. Then $g^{-1}$ is an isometry of a finite subspace of $\mathrm{M}$ to a finite subspace of $\mathrm{M}$, which has, because $\mathrm{M}$ is homogeneous, an extension, say $h$, to an isometry of $\mathrm{M}$ onto $\mathrm{M}$. The point $h(f(\mathfrak{k}))$ is a realization of $\mathfrak{k}$. 
\end{proof}

\begin{lem}\label{lem:realin5}
Let $\mathrm{M}=(M;\de)$ be a homogeneous metric space and $\mathrm{A}=(A;\de_\mathrm{A})$ a countable metric space with $A\cap M$ finite whose age is a subset of the age of $\mathrm{M}$ and for which $\de(x,y)=\de_\mathrm{A}(x,y)$ for all $x,y\in A\cap M$. Then there exists a realization of $\mathrm{A}$ in $\mathrm{M}$, that is a subset $B\subseteq M\setminus(A\cap M)$ for which there an isometry of $\mathrm{A}$ onto $\restrict{\mathrm{M}}{(B\cup (M\cap A))}$  which fixes $A\cap M$ pointwise.
\end{lem}
\begin{proof}
By induction on $A\setminus M$ or a recursive construction realizing \Kat functions step by step. 
\end{proof}

\begin{lem}\label{lem:univsub}
Let $\mathrm{M}=(M;\de)$ be a separable metric space and $T$ a countable subset of $M$. If\/ $\mathrm{M}$  realizes all of its restricted type functions then it  contains a countable dense subspace $\mathrm{S}$ with $T\subseteq S$, which realizes all of its restricted type functions. If\/ $\mathrm{M}$  realizes all of its \Kat functions then it  contains a countable dense subspace $\mathrm{S}$ with $T\subseteq S$, which realizes all of its \Kat functions.  
\end{lem}
\begin{proof}
Let $\mathrm{M}$ be separabe and realize all of its restricted type functions. For $A\subseteq M$  let $\spe(A)$ be the set of distances between points of $A$ and $\mathcal{K}(A)$ the set of restricted type  functions  $\mathfrak{k}$ of $\mathrm{M}$ with $\dom(\mathfrak{k})\subseteq A$.  

Let $S_0$ be a countable dense subset of $M$. If for $n\in \omega$ a countable set  $S_n$  has been determined, choose a realization $\bar{\mathfrak{k}}$ for every restricted type function $\mathfrak{k}\in \mathcal{K}(S_n)$. Let $S_{n+1}=S_n\cup \{\bar{\mathfrak{k}}\mid \mathfrak{k}\in \mathcal{K}(S_n)\}$. The set $S_{n+1}$ is countable because $\mathcal{K}(S_n)$ is countable. Then $S=\bigcup_{n\in \omega}S_n$ is countable and every restricted type function $\mathfrak{k}\in \mathcal{K}(S)$ has a realization in $S$.  

The proof in the case $\mathrm{M}$ homogeneous is analogues. 
\end{proof}
Hence we obtain from  Corollary \ref{cor:realin3}:

\begin{cor}\label{cor:univsub}
Every separable universal metric space $\mathrm{M}=(M;\de)$  contains a countable dense universal subspace. Every separable homogeneous metric space $\mathrm{M}=(M;\de)$  contains a countable dense homogeneous subspace.
\end{cor}

\begin{thm}\label{thm:realint}
Every  \Kat function of a homogeneous metric space $\mathrm{M}$ has a realization in $\mathrm{M}$. Every restricted type function of a universal metric space $\mathrm{M}$ has a realization in $\mathrm{M}$.

If a metric space $\mathrm{M}$ is countable and every \Kat function has a realization in $\mathrm{M}$ then $\mathrm{M}$ is homogeneous. If a metric space $\mathrm{M}$ is countable and every restricted type function has a realization in $\mathrm{M}$ then $\mathrm{M}$ is universal.

If a metric space $\mathrm{M}$ is complete and separable and every \Kat  function has a realization in $\mathrm{M}$ then $\mathrm{M}$ is homogeneous. If a metric space $\mathrm{M}$ is complete and separable and every restricted type function has a realization in $\mathrm{M}$ then $\mathrm{M}$ is an Urysohn metric space.
\end{thm}
\begin{proof}
On account of Corollary \ref{cor:realin3} and Lemma \ref{lem:realin4} it remains to consider the case that $\mathrm{M}$ is complete and separable.

Let $F$ be a finite subset of $M$ and $f$ an isometry of $\restrict{\mathrm{M}}{F}$  into $\mathrm{M}$. Lemma \ref{lem:univsub} yields a dense countable subspace $\mathrm{S}=(S;\de)$  of $\mathrm{M}$, with $F\subseteq S$,  which in the case of \Kat functions realizes all of its \Kat functions and hence is homogeneous on account of Corollary~\ref{cor:realin3}. It follows that there is an extension $g$ of $f$ to an isometry  of $\mathrm{S}$ onto $\mathrm{S}$. Because $\mathrm{M}$ is complete the isometry $g$ has an extension to an isometry of $\mathrm{M}$ to $\mathrm{M}$. It follows that $\mathrm{M}$ is homogeneous.  

In the case of reduced type functions the metric space $\mathrm{M}$ is homogeneous as well because because then every reduced type function is a \Kat function. Let $\mathrm{N}=(N;\de)$ be a separable metric space with $\dist(\mathrm{N})\subseteq \dist(\mathrm{M})$. Let $T$ be a countable dense subset of $N$. According to Lemma \ref{lem:realin1} there exists an isometry $f$ of $\restrict{\mathrm{N}}{T}$ into $\mathrm{M}$, which because $\mathrm{M}$ is complete, has an extension to an isometry of $\mathrm{N}$ into $\mathrm{M}$. Hence $\mathrm{M}$ is Urysohn.

\end{proof}

\begin{cor}\label{cor:realint}
Every separable, complete and universal metric space $\mathrm{M}$ is an Urysohn metric space.
\end{cor}
\begin{proof}
It follows from Theorem \ref{thm:realint} that $\mathrm{M}$ realizes all of its reduced type functions and hence it follows again from Theorem \ref{thm:realint} that $\mathrm{M}$ is Urysohn.
\end{proof}

\begin{thm}\label{thm:Uryuni}
Any  two homogeneous and separable and complete metric spaces   $\mathrm{M}$ and $\mathrm{N}$ with the same age are isometric. Any two Urysohn metric spaces $\mathrm{M}$ and $\mathrm{N}$ with $\dist(\mathrm{M})=\dist(\mathrm{N})$ are isometric.
\end{thm}
\begin{proof}
Let $S_0$ be a countable dense subset of $M$ and $T_0$ a countable dense subset of $N$. There exists an isometry $f$ of $\restrict{\mathrm{M}}{S_0}$ into $\mathrm{N}$ and then a dense countable homogeneous subspace $\mathrm{T}=(T;\de)$ of $\mathrm{N}$ with $f[S_0]\cup T_0\subseteq N$. There exists an isometry $g$ of $\mathrm{T}$ into $\mathrm{M}$ with $g(f(x))=x$ for all $x\in S_0$. Then $g[T]$ is dense in $M$ and because $\mathrm{M}$ and $\mathrm{N}$ are complete there exists an extension of $g$ to an isometry of $\mathrm{N}$ onto $\mathrm{M}$. 
\end{proof}

\begin{defin}\label{defin:amalg}
A pair of metric spaces $(\mathrm{A},\mathrm{B})$ of metric spaces is an {\em amalgamation instance} if\/ $\de_\mathrm{A}(x,y)=\de_\mathrm{B}(x,y)$ for all $x,y\in A\cap B$. Then $\amalg(\mathrm{A},\mathrm{B})$ is the set of metric spaces  with:
\[
\amalg(\mathrm{A},\mathrm{B})=\{\mathrm{C}=(A\cup B; \de_\mathrm{C})\mid\text{$\restrict{\mathrm{C}}{A}=\mathrm{A}$ and $\restrict{\mathrm{C}}{B}=\mathrm{B}$}\}.
\]
For $R\subseteq \Re_{\geq 0}$ let 
\[
\amalg_R(\mathrm{A},\mathrm{B})=\{\mathrm{C}\in \amalg(\mathrm{A},\mathrm{B})\mid \dist(\mathrm{C})\subseteq R\}.
\]
\end{defin}

\begin{defin}\label{defin:age}
An {\em age of metric spaces} is a class of finite metric spaces closed under subspaces and isometric copies and which is updirected, that is for all metric spaces $\mathrm{A}$ and $\mathrm{B}$ in the class exists a metric space $\mathrm{C}$ in the class which isometrically embeds both spaces $\mathrm{A}$ and $\mathrm{B}$. An age is {\em countable} if it has countably many isometry classes.
\end{defin}

\begin{defin}\label{defin:Fraclass}
A {\em \Fra class} $\mathcal{A}$ of metric spaces is a countable age of metric spaces which is closed under amalgamation, that is for all amalgamation instances $(\mathrm{A},\mathrm{B})$ with $\mathrm{A},\mathrm{B}\in \mathcal{A}$ exists a metric space $\mathrm{C}\in \mathcal{A}\cap \amalg(\mathrm{A},\mathrm{B})$.
\end{defin}

\begin{thm}\label{thm:Fraisse}[\Fra\hskip-5pt]
For every \Fra class $\mathcal{A}$ of metric spaces exists a unique countable homogeneous metric space $\boldsymbol{U}_{\hskip-2pt\mathcal{A}}$, the {\em \Fra limit} of $\mathcal{A}$, whose age is equal to $\mathcal{A}$.
\end{thm}
Note that Theorem \ref{thm:Fraisse} implies that two countable universal metric spaces with set of distances are isometric and together with Theorem~\ref{thm:Uryuni} that any two homogeneous and  countable or separable and complete metric spaces with the same age are isometric. Hence we can define:
\begin{defin}\label{defin:boldu}
An Urysohn metric space or countable universal metric space with set of distances equals to $R$ will be dinoted by $\UR$. A homogenous metric space with age $\mathcal{A}$ which is countable or separable and complete will be denoted by $\boldsymbol{U}_{\hskip-2pt{\mathcal{A}}}$.
\end{defin}

\section{The 4-values condition}

A triple $(a,b,c)$ of non negative numbers is {\em metric} if $a\leq b+c$ and $b\leq a+c$ and $c\leq a+b$. Note that a triple $(a,b,c)$ is metric if and only if $|a-b|\leq c\leq a+b$. For $(a,b,c,d)$ a quadruple of numbers and $x$ a number write $x\leadsto(a,b,c,d)$ for: The triples $(x,a,b)$ and $(x,c,d)$ are metric and $a\geq \max\{b,c,d\}$.

\begin{lem}\label{lem:triv4}
 If $x\leadsto (a,b,c,d)$ then $|a-d|\leq b+c$  and $|b-c|\leq  a+d$.
 \end{lem}
 \begin{proof}
If   $a\geq b$ then $|b-c|\leq  a+d$ and $a-b\leq x\leq c+d$ implying $|a-d|\leq b+c$. If $b\geq a$ then $|a-d|\leq b+c$ and $b-a\leq x \leq c+d$ implying $|b-c|\leq a+d$.
\end{proof}
 
\begin{defin}\label{defin:4valf}
For $R\subseteq \Re_{\geq 0}$ let $\boldsymbol{Q}(R)$ be the set of quadruples $(a,b,c,d)$ of numbers in $R$  for which there exists a number $x\in R$ with $x\leadsto(a,b,c,d)$.
\end{defin}

\begin{defin}\label{defin:4val}
A set  $R\subseteq \Re_{\geq 0}$ satisfies the {\em 4-values condition} if $(a,d,c,b)\in \boldsymbol{Q}(R)$ for every quadruple $(a,b,c,d)\in \boldsymbol{Q}(R)$   
\end{defin}
Note: The set $R$ satisfies the 4-values condition if and only if for every quadruple $(a,b,c,d)\in\boldsymbol{Q}(R)$ there exists a number $y\in R$ so that $y\leadsto (a,d,c,b)$ that is  the triples $(d,a,y)$ and $(b,c,y)$ are metric. 

\begin{lem}\label{lem:4val}
A set  $R\subseteq \Re_{\geq 0}$ satisfies the {\em 4-values condition} if and only if for any two metric spaces of the form $\mathrm{A}=(\{p,v,w\};\de_\mathrm{A})$ and $\mathrm{B}=(\{q,v,w\};\de_\mathrm{B})$ with $\dist(\mathrm{A})\subseteq R$ and $\dist(\mathrm{B})\subseteq R$ and $R\ni x=\de_\mathrm{A}(v,w)=\de_\mathrm{B}(v,w)$ the set\/ $\amalg_R(\mathrm{A},\mathrm{B})\not=\emptyset$.
\end{lem}
\begin{proof}
Let $R$ satisfy the 4-values condition and assume that the spaces $\mathrm{A}$ and $\mathrm{B}$ with $x=\de_\mathrm{A}(v,w)=\de_\mathrm{B}(v,w)$ are given. Let:
\begin{align}\label{cond:4val}
a=\de_\mathrm{A}(p,v)\, \, \, \, \,  b=\de_\mathrm{A}(p,w)\, \, \, \, \,  c=\de_\mathrm{B}(q,w)\, \, \, \, \,  d=\de_\mathrm{B}(q,v).
\end{align}
Then $x\leadsto(a,b,c,d)$. Hence there is a number $y\in R$ with $y\leadsto (a,d,c,b)$ implying that the space $\mathrm{C}=(\{p,q,v,w\}; \de_\mathrm{C})\in \amalg_R(\mathrm{A},\mathrm{B})$ with $\de_{\mathrm{C}}(p,q)=y$. 

For the other direction of the proof let $x,a,b,c,d\in R$ and $x\leadsto (a,b,c,d)$. Let then $\mathrm{A}=(\{p,v,w\}; \de_\mathrm{A})$ and $\mathrm{B}=(\{q,v,w\};\de_\mathrm{B})$ be metric spaces with distances  as in (\ref{cond:4val}) and with $x=\de_\mathrm{A}(v,w)=\de_\mathrm{B}(v,w)$. Let $\mathrm{C}=(\{p,q,v,w\}; \de_\mathrm{C})\in \amalg_R(\mathrm{A},\mathrm{B})$. Then $R\ni y\leadsto (a,d,c,b)$ for $y=\de_\mathrm{C}(p,q)$.
 \end{proof}
 
 \begin{defin}\label{defin:interval}
 For two metric spaces $\mathrm{A}=(\{p,v,w\};\de_\mathrm{A})$ and $\mathrm{B}=(\{q,v,w\};\de_\mathrm{B})$ with distances as in (\ref{cond:4val})  let: 
 \[
 \mathfrak{u}(\mathrm{A},\mathrm{B})=\max\{|a-d|,|b-c|\}\, \, \, \, \, \, \, \mathfrak{l}(\mathrm{A},\mathrm{B})=\min\{a+d,b+c\}
 \]
 \end{defin}

 \begin{lem}\label{lem:interval}
 Let $\mathrm{A}=(\{p,v,w\};\de_\mathrm{A})$ and $\mathrm{B}=(\{q,v,w\};\de_\mathrm{B})$ be two metric spaces with $\de_\mathrm{A}(v,w)=\de_\mathrm{B}(v,w)$, then $\mathfrak{u}(\mathrm{A},\mathrm{B})\leq \mathfrak{l}(\mathrm{A},\mathrm{B})$ and 
\[
[\mathfrak{u}(\mathrm{A},\mathrm{B}), \mathfrak{l}(\mathrm{A},\mathrm{B})]=\{\de_\mathrm{C}(p,q)\mid \mathrm{C}\in \amalg(\mathrm{A},\mathrm{B})\}.
\] 
Let $R\subseteq \Re_{\geq 0}$ satisfy the 4-values condition and  $\dist(\mathrm{A})\cup \dist(\mathrm{B})\subseteq R$. Then there exists a number $y\in R\cap[\mathfrak{u}(\mathrm{A},\mathrm{B}), \mathfrak{l}(\mathrm{A},\mathrm{B})]$.
\end{lem}
 \begin{proof}
Let  the numbers $(a,b,c,d)$ be given by Condition (\ref{cond:4val}). 
 
The inequality $\mathfrak{u}(\mathrm{A},\mathrm{B})\leq \mathfrak{l}(\mathrm{A},\mathrm{B})$ follows from Lemma  \ref{lem:triv4}. If $y\in [\mathfrak{u}(\mathrm{A},\mathrm{B}), \mathfrak{l}(\mathrm{A},\mathrm{B})]$ then $|a-d|\leq y \leq a+d$ and $|b-c|\leq y\leq b+c$ and hence the triples $(y,a,d)$ and $(y,b,c)$ are metric. If $y=\de_\mathrm{C}(p,q)$ for  $\mathrm{C}\in \amalg(\mathrm{A},\mathrm{B}$, then the triples $(y,a,d)$ and $(y,b,c)$ are metric and hence $|a-d|\leq y \leq a+d$ and $|b-c|\leq y\leq b+c$. 

If $R$ satisfies the 4-values condition and  $\dist(\mathrm{A})\cup \dist(\mathrm{B})\subseteq R$ it follows from Lemma \ref{lem:4val} that  $ \amalg_R(\mathrm{A},\mathrm{B})\not=\emptyset$ and hence $\emptyset\not= \{\de_\mathrm{C}(p,q)\mid \mathrm{C}\in \amalg_R(\mathrm{A},\mathrm{B})\}=R\cap \{\de_\mathrm{C}(p,q)\mid \mathrm{C}\in \amalg(\mathrm{A},\mathrm{B})\}=R\cap [\mathfrak{u}(\mathrm{A},\mathrm{B}), \mathfrak{l}(\mathrm{A},\mathrm{B})]$.
 \end{proof}

\begin{lem}\label{lem:alamgam1}
Let the set  $R\subseteq \Re_{\geq 0}$ satisfy the 4-values condition and let  $\big(\mathrm{A}=(A;\de_\mathrm{A}),\mathrm{B}=(B;\de_\mathrm{B})\big)$ with $\dist(\mathrm{A})\cup \dist(\mathrm{B})\subseteq R$ be an amalgamation instance. Let $A\setminus B=\{p\}$ and $B\setminus A=\{q\}$. 
 
 Then $\amalg_R(\mathrm{A},\mathrm{B})\not=\emptyset$ if $A\cup B$ is finite or if $R$ is closed. 
 \end{lem}
 \begin{proof}
For $v,w\in A\cap B$ let $\mathrm{A}_{v,w}=\restrict{\mathrm{A}}{\{p,v,w\}}$ and $\mathrm{B}_{v,w}=\restrict{\mathrm{B}}{\{q,v,w\}}$. Note that $(\mathrm{A}_{v,w},\mathrm{B}_{v,w})$ is an amalgamation instance. We have to prove that there is a number $y\in R$ so that the premetric space $\mathrm{C}=(A \cup B; \de)$ with $\de(p,q)=y$ and $\de(p,v)=\de_\mathrm{A}(p,v)$ and $\de(q,v)=\de_\mathrm{B}(q,v)$ and $\de(v,w)=\de_\mathrm{A}(v,w)=\de_\mathrm{B}(v,w)$ for all $v,w\in A\cap B$ is a metric space. That is, for 
\[
\mathcal{S}:=\bigcap_{v,w\in A\cap B}[\mathfrak{u}(\mathrm{A}_{v,w},\mathrm{B}_{v,w}), \mathfrak{l}(\mathrm{A}_{v,w},\mathrm{B}_{v,w})],
\]
we have to prove, according to  Lemma \ref{lem:interval}, that $\mathcal{S}\cap R\not=\emptyset$. Let 
\[
\hat{\mathfrak{u}}:=\sup\{|\de_\mathrm{A}(p,v)-\de_\mathrm{B}(q,v)| \mid v\in A\cap B\},\hskip 4 pt \hat{\mathfrak{l}}:=\inf\{ \de_\mathrm{A}(p,v)+\de_\mathrm{B}(q,v)\}.
\]
Then $\hat{\mathfrak{u}}\leq \hat{\mathfrak{l}}$  according to Lemma \ref{lem:triv4} and $[\hat{\mathfrak{u}},\hat{\mathfrak{l}}]\subseteq S$ according to Definition \ref{defin:interval}. 

Let  $R$ be  closed. There exists for every $\epsilon>0$ a point $v\in A\cap B$ with $\hat{\mathfrak{u}}-\epsilon<|\de_\mathrm{A}(p,v)-\de_\mathrm{B}(q,v)|\leq \hat{\mathfrak{u}}$ and a point $w\in A\cap B$ with $\hat{\mathfrak{l}}\leq \de_\mathrm{A}(p,w)+\de_\mathrm{B}(q,w)\leq \hat{\mathfrak{l}}+\epsilon$. Then:
\begin{align*}
&\hat{\mathfrak{u}}-\epsilon<|\de_\mathrm{A}(p,v)-\de_\mathrm{B}(q,v)|\leq\mathfrak{u}(\mathrm{A}_{v,w},\mathrm{B}_{v,w})\leq \hat{\mathfrak{u}}\leq \\
&\leq\hat{\mathfrak{l}} \leq \de_\mathrm{A}(p,w)+\de_\mathrm{B}(q,w)\leq \mathfrak{l}(\mathrm{A}_{v,w},\mathrm{B}_{v,w})<\hat{\mathfrak{l}}+\epsilon.
\end{align*}
Because $R$ satisfies the 4-values condition it follows that for every $\epsilon>0$ there exists a number $y_\epsilon\in R$ with 
\[
\hat{\mathfrak{u}}-\epsilon<\mathfrak{u}(\mathrm{A}_{v,w})\leq y_\epsilon\leq \mathfrak{l}(\mathrm{A}_{v,w},\mathrm{B}_{v,w}) <\hat{\mathfrak{l}}+\epsilon.
\]
Hence, because $R$ is closed, there exists a number $y\in R\cap[\hat{\mathfrak{u}},\hat{\mathfrak{l}}]\subseteq R\cap S$.

If $A\cup B$ is finite let $v\in A\cap B$ such that $\hat{\mathfrak{u}}=|\de_\mathrm{A}(p,v)-\de_\mathrm{B}(q,v)|$ and $w\in A\cap B$ such that $\hat{\mathfrak{l}}=\de_\mathrm{A}(p,w)=\de_\mathrm{B}(q,w)$. Then:
\[
\hat{\mathfrak{u}}=\mathfrak{u}(\mathrm{A}_{v,w},\mathrm{B}_{v,w})\leq \mathfrak{l}(\mathrm{A}_{v,w},\mathrm{B}_{v,w})=\hat{\mathfrak{l}}.
\]
Because $R$ satisfies the 4-values condition there exists a number $y\in R\cap [\hat{\mathfrak{u}},\hat{\mathfrak{l}}]\subseteq R\cap \mathcal{S}$. 
 \end{proof}

\begin{lem}\label{lem:alamgam2} 
Let the set  $R\subseteq \Re_{\geq 0}$ satisfy the 4-values condition and let  $\big(\mathrm{A}=(A;\de_\mathrm{A}),\mathrm{B}=(B;\de_\mathrm{B})\big)$ with $\dist(\mathrm{A})\cup \dist(\mathrm{B})\subseteq R$ be an amalgamation instance. Let $A\setminus B=\{p\}$. 
 
 Then $\amalg_R(\mathrm{A},\mathrm{B})\not=\emptyset$ if $A\cup B$ is finite or if $R$ is closed. 
\end{lem}
\begin{proof}
Note that for $A\cap B\subseteq C\subseteq B$ and $\mathrm{C}=\restrict{\mathrm{B}}{C}$ the pair $(\mathrm{A},\mathrm{C})$ is an amalgamation instance. Let:
\[
\mathcal{M}=\bigcup_{A\cap B\subseteq C\subseteq B}\amalg_R(\mathrm{A},\restrict{\mathrm{B}}{C}).
\]
Then $(\mathcal{M};\preceq)$ is a partial order for $\mathrm{L}=(L; \de)\preceq \mathrm{N}=(N;\de)$ if $L\subseteq N$ and $\restrict{\mathrm{N}}{L}=\mathrm{L}$. Every chain in the partial order $(\mathcal{M};\preceq)$ has an upper bound and hence using Zorn's Lemma the partial order $(\mathcal{M};\preceq)$ has a maximal element $\mathrm{M}=(M;\de_\mathrm{M})$. If $M=A\cup B$ then $\mathrm{M}\in \amalg_R(\mathrm{A},\mathrm{B})$. Otherwise let $b\in (A\cup B)\setminus M$ and let $D=(M\setminus\{p\})\cup \{b\}$ and $\mathrm{D}=\restrict{\mathrm{B}}{D}$. Lemma \ref{lem:alamgam1} applied to the amalgamation instance $(\mathrm{M},\mathrm{D})$ results in a metric space contradicting the maximality of $\mathrm{M}$.
\end{proof}

\begin{thm}\label{thm:alamgam3} 
Let the set  $R\subseteq \Re_{\geq 0}$ satisfy the 4-values condition and let  $\big(\mathrm{A}=(A;\de_\mathrm{A}),\mathrm{B}=(B;\de_\mathrm{B})\big)$ with $\dist(\mathrm{A})\cup \dist(\mathrm{B})\subseteq R$ be an amalgamation instance.
 
 Then $\amalg_R(\mathrm{A},\mathrm{B})\not=\emptyset$ if $A\cup B$ is finite or if $R$ is closed. 
\end{thm}
\begin{proof}
Let:
\[
\mathcal{M}=\bigcup_{A\cap B\subseteq C\subseteq B}\amalg_R(\mathrm{A},\restrict{\mathrm{B}}{C}).
\]
Then $(\mathcal{M};\preceq)$ is a partial order for $\mathrm{L}=(L; \de)\preceq \mathrm{N}=(N;\de)$ if $L\subseteq N$ and $\restrict{\mathrm{N}}{L}=\mathrm{L}$. Every chain in the partial order $(\mathcal{M};\preceq)$ has an upper bound and hence using Zorn's Lemma the partial order $(\mathcal{M};\preceq)$ has a maximal element $\mathrm{M}=(M;\de_\mathrm{M})$. If $M=A\cup B$ then $\mathrm{M}\in \amalg_R(\mathrm{A},\mathrm{B})$. Otherwise let $b\in (A\cup B)\setminus M$ and let $D=(M\cap B)\cup \{b\}$ and $\mathrm{D}=\restrict{\mathrm{B}}{D}$. Lemma \ref{lem:alamgam2} applied to the amalgamation instance $(\mathrm{M},\mathrm{D})$ results in a metric space contradicting the maximality of $\mathrm{M}$.
\end{proof}

\begin{thm}\label{thm:4nec}
The set of distances of a universal metric space satisfies the 4-values condition. 
\end{thm}
\begin{proof}
Let $\mathrm{M}=(M;\de)$ be a universal metric space with $R=\dist(\mathrm{M})$. Let $\mathrm{A}=(\{p,v,w\};\de_\mathrm{A})$ and $\mathrm{B}=(\{q,v,w\};\de_\mathrm{B})$ with $\dist(\mathrm{A})\subseteq R$ and $\dist(\mathrm{B})\subseteq R$ and $R\ni x=\de_\mathrm{A}(v,w)=\de_\mathrm{B}(v,w)$. There exists an isometric copy with points $\{p',v',w'\}\subseteq M$ in $\mathrm{M}$. Let $\mathfrak{t}$ be the restricted type function with $\dom(\mathfrak{t})=\{v',w'\}$ and with $\mathfrak{t}(v')=\de_\mathrm{B}(q,v)$ and $\mathfrak{t}(w')=\de_\mathrm{B}(q,w)$. Let $q'$ be a realization of $\mathfrak{t}$. Then the metric space $\mathrm{C}=(\{p,v,w,q\};\de_\mathrm{C})$ with $\restrict{\mathrm{C}}{\{p,v,w\}}=\mathrm{A}$ and $\restrict{\mathrm{C}}{\{q,v,w\}}=\mathrm{B}$ and $\de_\mathrm{C}(p,q)=\de(p',q')$ is a metric space in $\amalg_R(\mathrm{A},\mathrm{B})$. Hence the Theorem follows from Lemma \ref{lem:4val}.
\end{proof}

\begin{defin}\label{defin:finf}
Let $R\subseteq \Re_{\geq 0}$, then $\mathcal{F}_R$ is the class of finite metric spaces $\mathrm{M}$ with $\dist(\mathrm{M})\subseteq R$.
\end{defin}

\begin{thm}\label{thm:basic1}
Let $0\in R\subseteq \Re_{\geq 0}$ be a countable set of numbers which satisfies the 4-values condition. Then there exists a countable universal metric space $\UR$.

If there exists a countable universal metric space $\UR$ then $R$ satisfies the 4-values condition. 
\end{thm}
\begin{proof}
The class  $\mathcal{F}_R$ of finite metric spaces is closed under isometric copies and substructures and it follows from Theorem \ref{thm:alamgam3} and Definition \ref{defin:amalg} that $\emptyset\not=\amalg_R(\mathrm{A},\mathrm{B})\subseteq \mathcal{F}_R$ for all $\mathrm{A},\mathrm{B}\in \mathcal{F}_R$. Hence $\mathcal{F}_R$ is updirected and closed under amalgamation and hence a \Fra class. According to Theorem \ref{thm:Fraisse} there exists a countable homogeneous metric space $\boldsymbol{U}_{\hskip-2pt\mathcal{F}_R}$ whose age is equal to $\mathcal{F}_R$. It follows that $\boldsymbol{U}_{\hskip-2pt\mathcal{F}_R}$ is the countable universal metric space $\UR$. 
\end{proof}

\begin{lem}\label{lem:4valdenseR}
Let  $R\subseteq \Re_{\geq 0}$ be a set of  numbers which satisfies the 4-values condition. For every countable subset $T$ of $R$ exists a dense   countable subset $C\supseteq T$ of $R$ which satisfies the 4-values condition. 
\end{lem}
\begin{proof}
Let $S\supseteq T$ be a countable dense subset of $R$. There are countably many instances of the form $x\leadsto (a,b,c,d)$  with numbers  in $S$. Because $R$ does satisfy the 4-values condition  there is a countably set $S'\subseteq R$ so that for all those quadruples there is a $y\in S'$ with $y\leadsto (a,d,c,b)$. Repeating this process countably often leads to a countable subset $C\subseteq R$ which satisfies the 4-values condition.
\end{proof}

In order to verify the 4-values condition the following Lemma is often useful.
\begin{lem}\label{lem:verify4}
If   $a\leq b+c$ or $a\leq b+d$ or $a\leq c+d$  and   $(a,b,c,d)\in \boldsymbol{Q}(R)$ then  there exists $y\in \{a,b,c,d\}$ with $y\leadsto (a,d,c,b)$.
\end{lem}
\begin{proof}
If $a\leq b+c$ then $a\leadsto (a,d,c,b)$. If $a\leq b+d$ then $b\leadsto (a,d,c,b)$ unless $2b<c$ in which case $c\leadsto (a,d,c,b)$. If $a\leq c+d$ then $c\leadsto (a,d,c,b)$ unless $2c<b$ in which case $b\leadsto (a,d,c,b)$. 
\end{proof}
Hence, in order to verify that $R$ satisfies the 4-values condition, it suffices to consider quadruples for which $a$ is larger than the sum of any two  of the other three numbers.

\section{Completion of universal metric spaces}

\begin{defin}\label{defin:join}
Let $\mathrm{A}=(A;\de_\mathrm{A})$ and $\mathrm{B}=(B;\de_\mathrm{B})$ be two metric spaces with $A=\{a_i\mid i\in m\in \omega\}$ and $B=\{b_i\mid i\in m\in \omega\}$ and $A\cap B=\emptyset$. A metric space $\mathrm{P}=(A\cup B;\de_\mathrm{P})$ with $\restrict{\mathrm{P}}{A}=\mathrm{A}$  and $\restrict{\mathrm{P}}{B}=\mathrm{B}$ is an $h$-join of $\mathrm{A}$ and $\mathrm{B}$ if $\de_\mathrm{P}(a_i,b_i)<h$ for all $i\in m$. (The $h$-join $\mathrm{P}$ depends explicitly on the enumeration of $A$ and $B$.)
\end{defin}

\begin{lem}\label{lem:join1}
Let $R\subseteq \Re_{\geq 0}$ satisfy the 4-values condition and have $0$ as a limit and let $\mathcal{F}_R$ be the class of finite metric spaces $\mathrm{F}$ with $\dist(\mathrm{F})\subseteq R$.

Let $\mathrm{A}=(A;\de_\mathrm{A})$ and $\mathrm{B}=(B;\de_\mathrm{B})$ be two metric spaces in $\mathcal{F}_R$ for which $A=\{a_0,a_1,a_2,\dots,a_m\}$ and $B=\{b_0,b_1,b_2,\dots,b_m\}$ and $A\cap B=\emptyset$. Let $A'=\{a_i\mid i\in m\}$ and $B'=\{b_i\mid i\in m\}$ and $\mathrm{A}',\mathrm{B}'$ metric spaces with $\mathrm{A}'=\restrict{\mathrm{A}}{A'}$ and $\mathrm{B}'=\restrict{\mathrm{B}}{B'}$. Let $k\geq \max\{|\de_\mathrm{A}(a_m,a_i)-\de_\mathrm{B}(b_m,b_i)|\mid i\in m\}$ and $h\in R$ and $l\in R$ with :
\[
l+k\leq h\leq \min\big(\{\de_\mathrm{A}(a_m,a_i)\mid i\in m\}\cup \{\de_\mathrm{B}(b_m,b_i)\mid i\in m\}\big).
\]
Then if there exists an $l$-join $\mathrm{Q}=(A',B');\de_\mathrm{Q})\in \mathcal{F}_R$ of $\mathrm{A}'$ and $\mathrm{B}'$, there exists an $h$-join $\mathrm{P}\in \mathcal{F}_R$ of $\mathrm{A}$ and $\mathrm{B}$.
\end{lem}
\begin{proof}
There exist, according to Theorem \ref{thm:alamgam3}, a metric space $\mathrm{A^\ast}=(A^\ast;\de_{\mathrm{A}^\ast})\in \amalg_R(\mathrm{A},\mathrm{Q})$ and a metric space  $\mathrm{B^\ast}=(A^\ast;\de_{\mathrm{B}^\ast})\in \amalg_R(\mathrm{B},\mathrm{Q})$. Note  $|\de_\mathrm{A}(a_m,a_i)-\de_{\mathrm{A}^\ast}(a_m,b_i)|<l$ and $|\de_\mathrm{B}(b_m,b_i)-\de_{\mathrm{B}^\ast}(b_m,a_i)|<l$ for all $i\in m$. Let $\mathrm{P}=(A\cup B;\de)$ be the premetric space with $\restrict{\mathrm{P}}{(A\cup B')}=\mathrm{A}^\ast$ and $\restrict{\mathrm{P}}{(B\cup A')}=\mathrm{B}^\ast$ and with $\de(a_m,b_m)=h$.

In order to see that $\mathrm{P}$ is a metric space we have to check the triples of the form $(a_m,b_m,a_i)$ and $(a_m,b_m, b_i)$ for all $i\in m$. Indeed:
\begin{align*}
&|\de(a_m,a_i)-\de(b_m,a_i)\leq \\
&|\de(a_m,a_i)-\de(b_m,b_i)|+|\de(b_m,b_i)-\de(b_m,a_i)|\leq k+l\leq h=\de(a_m,b_m).
\end{align*}
This verifies that the triple $(a_m,b_m,a_i)$ is metric because $h\leq \de(a_m,a_i)$ and hence the distance $\de(a_m,b_m)$ is not larger than the other two distances in the triple $(a_m,b_m,a_i)$. Similar, the triangles of the form $(a_m,b_m,b_i)$ are metric.  
\end{proof}

\begin{lem}\label{lem:join2}
Let $R\subseteq \Re_{\geq 0}$ satisfy the 4-values condition and have $0$ as a limit. Then, for every $m\in \omega$ and $\boldsymbol{r}>0$ and $h\in \Re_{>0}$ exists a number $\gamma(h)<h$ so that:

For all metric spaces $\mathrm{A}=(\{a_i\mid i\in m\};\de_\mathrm{A})$ and $\mathrm{B}=(\{b_i\mid i\in m\};\de_\mathrm{B})$ in $\mathcal{F}_R$ with $|\de_\mathrm{A}(a_i,a_j)-\de_\mathrm{B}(b_i,b_j)|<\gamma(h)$ for all $i,j\in m$ and with $\min\big(\dist(\mathrm{A})\cup\dist(\mathrm{B})\big)\geq \boldsymbol{r}$ and $A\cap B=\emptyset$ exists an $h$-join of\/ $\mathrm{A}$ and $\mathrm{B}$ in $\mathcal{F}_R$.
\end{lem}
\begin{proof}
For every $0<x\in R$\/ let $0<x^\ast\in R$ with $2\cdot x^\ast<x$. Given $h$ and $\boldsymbol{r}$ let $0<h_{m-1}<\min\{h,\boldsymbol{r}\}$ and for all integers $m-1> i\geq 0$ let $h_i=h^\ast_{i+1}$ and $\gamma(h)=h_0$. The Lemma follows from Lemma \ref{lem:join1} via induction on $i$.
\end{proof}

Note: Let  $a,b,a',b'$ be four points in a metric space, then 
\begin{align}\label{align:4points}
&|\de(a,b)-\de(a',b')|\leq \\
&|\de(a,b)-\de(a,b')|+|\de(a,b')-\de(a',b')|\leq \de(b,b')+\de(a,a').\notag
\end{align}

\begin{thm}\label{thm:join3}
Let $0\in R\subseteq \Re_{\geq 0}$ be countable,  satisfy the 4-values condition and have $0$ as a limit.  Let  $\UR$ be  the countable universal metric space with $R$ as set of distances given by Theorem~\ref{thm:basic1}.  Then  $\mathrm{M}=(M;\de)$, the completion of $\UR$,   is homogeneous and separable and complete.   
\end{thm}
\begin{proof}
The universal space  $\UR=(U;\de)$ is dense in $\mathrm{M}$.   It follows from Theorem \ref{thm:realint} that $\mathrm{M}$ is homogeneous if every \Kat function of $\mathrm{M}$ has a realization in $\mathrm{M}$.

Let $\mathfrak{k}$ be a \Kat function of $\mathrm{M}$ with $\dom(\mathfrak{k})=\{a_i\mid i\in m\in \omega\}:=A$. There exists a subset $B=\{b_i\mid i\in m\}\cup \{b_m\}\subseteq M$ with $\de(a_i,a_j)=\de(b_i,b_j)$ and $\mathfrak{k}(a_i)=\de(b_m,b_i)$ for all $i,j\in m$. Let $k=\min\big(\dist(\Sp(\mathfrak{k}))\big)$. 

Let $0<e\in \Re$ and $B'=\{b_i'\mid 0\leq  i\leq  m\}\subseteq U$ with $\de(b_i,b_i')<e$. Then from Inequality (\ref{align:4points}):  $|\de(b_i,b_j)-\de(b_i',b_j')|<2e$ for all $0\leq i,j\leq m$. Note that if $e<\frac{k}{4}$ then $\min\{\de(b_i,b_j)\mid 0\leq i,j\leq m\}>\frac{k}{2}$. If there is a set of points $A'=\{a_i'\mid i\in m\}\subseteq U$ with $\de(a_i,a_i')<e$ then: $|\de(a_i,a_j)-\de(a_i',a_j')|<2e$ for all $i,j\in m$. Hence, because $\de(a_i,a_j)=\de(b_i,b_j)$: 
\[
\de(a_i',a_j')-\de(b_i',b_j')|<4e.
\]
It follows from Lemma \ref{lem:join2} that if $4e<\min(\gamma(h),k)$ then there exists an $h$-join $\mathrm{C}'=(\{B'\cup A'\};\de_{\mathrm{C}'})\in \mathcal{F}_R$  of $\restrict{\UR}{B'}$ and $\restrict{\UR}{A'}$. It follows from Lemma~\ref{lem:realin4} that there exists a realization $\mathrm{C}=(\{c_i\mid i\in m\};\de)$ of $\mathrm{C}'$ in $\UR$ with $\de(a_i',c_i)<h$ and $\de(c_i,c_j)=\de(b'_i,b'_j)$ for all $i,j\in m$. Also:
\begin{align}\label{align:conva}
\text{$\de(c_i,a_i)\leq \de(c_i,a_i')+\de(a_i',a_i)<h+e<2h$ for all $i\in m$}.
\end{align}

Let $(h_n;n\in \omega)$ and $(e_n;n\in \omega)$ be a sequence of numbers in $R$ with  $h_n>2\cdot h_{n+1}$ and $4e_n<\min(\gamma(\frac{1}{2}h_n),k)$. Then there exist sets of points: 
\begin{align}\label{align:exc}
&B_n'=\{b'_{n,i}\mid 0\leq i\leq m\} \text{ with $\de(b_{n,i}',b_i)<e_n<\frac{1}{2}h_n$ and}\notag \\ 
&C_n=\{c_{n,i}\mid i\in m\} \text{ with $\de(c_{n,i},a_i)<h_n$ and with}\notag \\
&\de(b_{n,i}',b_{n,j}')=\de(c_{n,i},c_{n,j}) \text{ and $\de(c_i,a_i)<h_n$ for all $i,j\in m$ and}\notag \\
&\de(b'_{n,m},b'_{n+1,m})\leq \de(b'_{n,m},b_{n,m})+\de(b_{n,m},b'_{n+1,m})<h_n.
\end{align}

Using the fact that \Kat functions of $\UR$ have realizations in $U$ construct recursively points $c_{n,m}\in U$ so that for all $i\in m$ and $n\in \omega$:
\[
\de(c_{n,m},c_{n,i})=\de(b'_{n,m},b'_{n,i})\, \, \, \, \, \, \,  \de(c_{n,m},c_{n+1,m})=\de(b'_{n,m},b'_{n+1,m})<h_n.
\]
That is,  the function $f$ with $f(b'_{n,i})=c_{n,i}$ for $0\leq i\leq m$ and $n\in \omega$ is an isometry of a subset of $U$ to a subset of $U$. 

It follows from Inequality (\ref{align:conva}) that for every  $i\in m$ the sequence $(c_{n,i})$ converges to $a_i$ and from Inequality (\ref{align:exc}) that the sequence $(c_{n,m})$ is Cauchy converging to, say $c_m$. For every $0\leq i\leq m$ the sequence $(b'_{n,i})$ converges to $b_i$ and hence for every $i\in m$:
\[
\lim_{n\to \infty}\de(b'_{n,m},b_{n,i})=\de(b_m,b_i)=\mathfrak{k}(a_i).
\]
It follows that $\de(c_m,a_i)=\lim_{n\to \infty}\de(c_{n,m},c_{n,i})=\mathfrak{k}(a_i)$ implying that $c_m$ is a realization of the \Kat function $\mathfrak{k}$. 

\end{proof}

\begin{thm}\label{thm:afterth}
The completion of a homogeneous  separable  metric space $\mathrm{M}$ which embeds isometrically  every finite metric space $\mathrm{F}$ with $\dist(\mathrm{F})\subseteq \dist(\mathrm{M})$   is homogeneous. 
\end{thm}
\begin{proof}
The space $\mathrm{M}$ contains a countable dense universal  subspace $\mathrm{N}$ according  to Corollary \ref{cor:univsub}.  Let $R=\dist(\mathrm{N})$. It follows from Theorem~\ref{thm:Uryuni} that we can take $\mathrm{N}$ to be the universal metric space $\UR$. The Theorem follows  from Theorem \ref{thm:join3} because the completion of $\UR$ is equal to the completion of $\mathrm{M}$.

\end{proof}

\begin{thm}\label{thm:embedwh}
Let $0\in R\subseteq \Re_{\geq 0}$ be countable,  satisfy the 4-values condition and have $0$ as a limit and let  $\mathrm{M}=(M;\de)$ be  the completion of $\UR=(U;\de)$. 

A finite metric space $\mathrm{A}=(A;\de_{\mathrm{A}})$ with $A=\{a_i\mid i\in m\}$ has an isometric embedding into $\mathrm{M}$ if and only if for ever $\epsilon>0$ there exists a metric space $\mathrm{B}=(B;\de_\mathrm{B})\in \mathcal{F}_R$ with $B=\{b_i\mid i\in m\}$ so that $|\de(a_i,a_j)-\de(b_i,b_j)|<\epsilon$ for all $i,j\in m$. 
\end{thm}
\begin{proof}
The condition is clearly necessary.

Let $k=\min\big(\dist(\mathrm{A})\big)$. Let  $(h_n)$ be a sequence of positive numbers in $R$ so that $h_0<\frac{k}{4}$ and $2h_{n+1}<h_n$ for all $n\in \omega$. Let $(e_n)$ be a sequence of positive  numbers in $R$ so that $\gamma(h_n)<2e_n$ and $e_{n+1}<e_n$, with $\gamma$ given by Lemma~\ref{lem:join2}. For $n\in \omega$ let  $\mathrm{B}_n=(B_n;\de_{\mathrm{B}_n})$ be a metric space with $B_n=\{b_{n,i}\mid i\in m\}\in \mathcal{F}_R$ and with 
\begin{align}\label{align:ineq23}
|\de_\mathrm{A}(a_i,a_j)-\de_{\mathrm{B}_n}(b_{n,i},b_{n,j})|<e_n<h_n.
\end{align}
Then:
\begin{align*}
|\de_{\mathrm{B}_{n+1}}(b_{n+1,i},b_{n+1,j})-\de_{\mathrm{B}_n}(b_{n,i},b_{n,j})|<e_{n+1}+e_n<2e_n.
\end{align*}
It follows from Lemma~\ref{lem:join2}  that there exists, for every $n\in\omega$,   an $h_n$-join $\mathrm{P}_n=(B_n\cup B_{n+1};\de_\mathrm{P})\in \mathcal{F}_R$ of $\mathrm{B}_n$ with $\mathrm{B}_{n+1}$.  

The space $\UR$ is universal and hence each of the finite metric spaces $\mathrm{P}_n$ has an isometric embedding into $\UR$. It follows from Lemma \ref{lem:realin5} via a recursive construction that there exist isometric copies $B_n'=\{b_{n,i}'\mid i\in m\}$ of the sets $B_n$ in $U$ so that, for all $n\in \omega$: 
\[
\de(b_{n,i}',b_{n+1,i}')=\de_{\mathrm{P}_n}(b_{n,i},b_{n+1,i})<h_n.
\]
It follows that for every $i\in m$ the sequence $(b_{n,i})$ is Cauchy and hence has a limit, say $b_i\in M$. Also for all $i,j\in m$: 
\[
\lim_{n\to\infty}\de(b'_{n,i},b'_{n,j})=\lim_{n\to\infty}\de(b_{n,i},b_{n,j})=\de_\mathrm{A}(a_i,a_j),
\] 
with the last equality implied by Inequality (\ref{align:ineq23}).
\end{proof}

\begin{cor}\label{cor:embedwh}
Let $0\in R\subseteq \Re_{\geq 0}$ be countable,  satisfy the 4-values condition and have $0$ as a limit and let  $\mathrm{M}=(M;\de)$ be  the completion of $\UR=(U;\de)$. Then $\dist(\mathrm{M})$ is the closure of $R$. The set of distances of the completion $\mathrm{N}$  of a universal separable metric space is a closed subset of $\Re$. 
\end{cor}
\begin{proof}
Let $\epsilon>0$ be given and $a$ in the closure of $\dist(\mathrm{M})$. There exists a number $b\in \dist(\mathrm{M})$ with $|a-b|<\frac{\epsilon}{2}$. There exists a number $c\in R$ with $|b-c|<\frac{\epsilon}{2}$. That $\dist(\mathrm{N})$ is closed follows as in the proof of Theorem \ref{thm:afterth}.
\end{proof}
Note that  in general the distance set of  the completion of a metric space need not be closed. (See Example \ref{ex:notclosed}.)

\begin{defin}\label{defin:rightspr}
For $R\subseteq \Re$ let 
\[R^{()}=\{x\in R\mid \exists \epsilon>0\, \, \big(\, (x,x+\epsilon)\cap R=\emptyset\big)\}.
\]
\end{defin}

\begin{lem}\label{lem:rightspr}
Let $R\subseteq \Re_{\geq 0}$, satisfy the 4-values condition and have 0 as a limit. If $\{x,y,z\}\subseteq R$ with $z=y+x$ and $x\in R^{()}$ then $z\in R^{()}$. 
\end{lem}
\begin{proof}
Let $\{x,y,z\}\subseteq R$ with $z=y+x$ and $x\in R^{()}$ and $\epsilon>0$ so that $(x,x+\epsilon)\cap R=\emptyset$ and let $0<\delta<\min\{\epsilon,x\}$. If $z\not\in R^{()}$ there exists $z<z'\in R$ with $z'-z<\delta$. Then $z\leadsto (z',\delta,x,y)$. If $R\ni u\leadsto (z',y,x,\delta)$ then the triple $(\delta,x,u)$ is metric and hence $u\leq x+\delta$, which implies, because $(x,\delta]\cap R=\emptyset$, that $u\leq x$. It follows that $u+y\leq x+y=z<z'$ and hence that the triple $(u,y,z')$ is not metric. In contradiction to $R$ satisfying the 4-values condition. 
\end{proof}

\begin{lem}\label{lem:rihtspr}
Let $R\subseteq \Re_{\geq 0}$, satisfy the 4-values condition and have 0 as a limit. If $\{x,y,z\}\subseteq R$ with $z=y+x$ and $\{x,y\}\subseteq R^{()}$ then both $x$ and $y$ are isolated points of $R$ and $z\in R^{()}$.
\end{lem}
\begin{proof}
It follows from Lemma \ref{lem:rightspr} that $z\in R^{()}$. Let $z=y+x$ and $\{x,y\}\subseteq R^{()}$. If, say $x$, is not isolated in $R$, let $\epsilon>0$ be such that $(y,y+\epsilon)\cap R=\emptyset$. Let $R\ni \delta<x$ with $0<\delta<\epsilon$ and let $0<u<x$ with $u\in R$ such that $x-u\leq \delta$. Note that  $x\leadsto (z,y,\delta,u)$.

If $R\ni r\leadsto (z,u,\delta,y)$ then $r\leq y+\delta$ because the triple $(r,y,\delta)$ is metric and hence $r\leq y$ because of the choice of $\delta$. Then $r+u\leq y+u<y+x=z$ and hence the triple $(r,u,z)$ is not metric in contradiction to $r\leadsto (z,u,\delta,y)$.
\end{proof}

\begin{lem}\label{lem:unemb5}
Let $0\in R\subseteq \Re_{\geq 0}$  satisfy the 4-values condition and have $0$ as a limit.  Let $S$ be a dense subset of $R$. 

Then there exists, for every metric space $\mathrm{A}=(A;\de_\mathrm{A})\in \mathcal{F}_R$ with $A=\{a_i\mid i\in m\}$ and every $\epsilon >0$, a metric space $\mathrm{B}=(B;\de_\mathrm{B})\in \mathcal{F}_S$ with $B=\{b_i\mid i\in m\}$ so that $|\de_\mathrm{A}(a_i,a_j)-\de_\mathrm{B}(b_i,b_j)|<\epsilon$ for all $i,j\in m$.
\end{lem}
\begin{proof}
Let:
\begin{align*}
&\Delta=\frac{1}{3}\min\{y+x-z\mid \text{ $z<y+x$ and $\{x,y,z\}\subseteq \dist(\mathrm{A})$}\}.
\end{align*}
Let $I\subseteq \dist(\mathrm{A})$ be the set of isolated points of $R$ which are elements of $\dist(\mathrm{A})$. Note that $I\subseteq S$. 

Let $E$ be the set of positive numbers in $\dist(\mathrm{A})\setminus R^{()}$ with $e_0<e_1<e_2<\dots<e_{n-1}<e_n$ an enumeration of $E$. Let $e_0<\hat{e}_0\in S$ so that $\hat{e}_0-e_0<\min\{\Delta,\epsilon\}$. The numbers $\hat{e}_i\in S$ are determined recursively so that always $\hat{e}_i>e_i$ and $\frac{1}{3}(\hat{e}_i-e_i)>\hat{e}_{i+1}-e_{i+1}$ for all indices $i\in n$. 

Let $K$ be the set of positive numbers in $\dist(\mathrm{A})\cap R^{()}$ which are not isolated and let $k_0>k_1>k_2>\dots>k_{r-1}>k_r$ be an enumeration of $K$. Let $k_0>\hat{k}_0\in S$  so that $k_0-\hat{k}_0<\frac{1}{3}(\hat{e}_n-e_n)$. The numbers $\hat{k}_i\in S$ are determined recursively so that always $k_i>\hat{k}_i$ and $\frac{1}{3}(k_i-\hat{k}_i)>k_{i+1}-\hat{k}_{i+1}$ for all indices $i\in r$. 

For every $x\in I$ let $\hat{x}=x$ and let $\hat{0}=0$. Note that the inequalities above imply for $\{x,y,z\}\subseteq \dist(\mathrm{A})$ with $x,y\not=0$ and $z=x+y$ and $x\in E$ or $y\in E$ that $\hat{z}\leq \hat{x}+\hat{y}$. 
\vskip 6pt
\noindent
CLAIM: If $(x,y,z)$ is a metric triple of numbers with entries in $\dist(\mathrm{A})$ then $(\hat{x},\hat{y},\hat{z})$ is a metric triples of numbers with entries in $S$. 

\noindent
Proof: Let $z\geq \max\{y,x\}$. It follows from the choice of $\Delta$ and the definition of the function $\hat{\, \, }$  that $\hat{z}\geq \max\{\hat{y},\hat{x}\}$ and if $z<y+x$ then $\hat{z}<\hat{y}+\hat{x}$ and hence that the triple $(\hat{z},\hat{y},\hat{x})$ is metric. 

Let $z=y+x$ with $x,y\not=0$. If at least one of $x$ and $y$ are in $E$ then $\hat{z}\leq \hat{x}+\hat{y}$. If both are not in $E$ then they are both in $R^{()}$ and it follows from Lemma \ref{lem:rihtspr} that both $x$ and $y$ are isolated in $R$ and $z\in R^{()}$. If $z$ is isolated then $\hat{z}=z=x+y=\hat{x}+\hat{y}$. If $z$ is not isolated then $\hat{z}<z=x+y=\hat{x}+\hat{y}$.\\
End proof of CLAIM. 

It follows that the premetric space $\mathrm{B}=(B;\de_\mathrm{B})$ with $B=\{b_i\mid i\in m\}$ and $\de_\mathrm{B}(b_i,b_j)=\hat{x}_{i,j}$ for $x_{i,j}=\de_{\mathrm{A}}(a_i,a_j)$ is a metric space with $|\de_\mathrm{A}(a_i,a_j)-\de_\mathrm{B}(b_i,b_j)|<\epsilon$ for all $i,j\in m$. 
\end{proof}

\begin{cor}\label{cor:unemb5}
Let $0\in R\subseteq \Re_{\geq 0}$  satisfy the 4-values condition and have $0$ as a limit.  Let $S$ be a dense subset of $R$. 

Then there exists, for every metric space $\mathrm{A}=(A;\de_\mathrm{A})\in \mathcal{F}_R$ with $A=\{a_i\mid i\in m\}$ and every $\epsilon >0$, a metric space $\mathrm{B}=(B;\de_\mathrm{B})\in \mathcal{F}_S$ with $B=\{b_i\mid i\in m\}$ so that $\de_\mathrm{A}(a_i,b_i)<\epsilon$ for all $i\in m$.
\end{cor}
\begin{proof}
Follows from Lemma \ref{lem:unemb5} and Lemma \ref{lem:join2}.
\end{proof}

\begin{lem}\label{lem:clUry0}
Let $\UR=(U;\de)$ be an Urysohn metric space for which 0 is a limit of $R$. Then $R$ is a closed subset of\/ $\Re$. 
\end{lem}
\begin{proof}
Let $c$ in the closure of $R$. Let $(a_n)$ be a sequence of numbers in $R$ converging monotonic to 0 with $a_0<\frac{1}{2}c$. Let $(c_n)$ be a sequence of numbers in $R$ for which $|c-c_n|$ converges monotonic to 0 with $|c-c_0|<\frac{1}{2}c$. For every $n\in \omega$ let $\bar{n}\in \omega$ with $n\leq \bar{n}$ and $2|c-c_{\bar{n}}|<a_n$ and $\bar{n}<\bar{m}$ for $n<m$.  

Let $\mathrm{V}=(V;\de)$ be a premetric space with  $V=\{v_n\mid n\in \omega\}\cup \{w\}$ so that $\de(v_n,v_m)=a_n$ for $n<m$ and $\de(v_n,w)=c_{\bar{n}}$. Then $\restrict{\mathrm{V}}(\{v_n\mid n\in\omega\})$ is a metric space and every triple $(a_n,\bar{n},\bar{m})$ is metric because $|c_{\bar{n}}-c_{\bar{m}}|\leq |c-c_{\bar{n}}|+|c-c_{\bar{m}}|\leq 2|c_{\bar{n}}-c|<a_n$ and $a_n<\bar{n}$. It follows that $\mathrm{V}$ is a metric space. 

There exists an isometric embedding $f$ of $\mathrm{V}$ into $\UR$. The sequence $\big(f(v_n)\big)$ is Cauchy  with limit, say $v\in U$. Then:   
\vskip 2pt
\noindent
$\de(f(w),v)=\lim_{n\to \infty}\de(f(w),f(v_n))=\lim_{n\to \infty}c_{f(n)}=\lim_{n\to \infty}c_n=c$.

\end{proof}

\begin{thm}\label{thm:haupt}
Let $0\in R\subseteq \Re_{\geq 0}$ with $0$ as a limit.  Then there exists an Urysohn metric space $\UR$ if and only if $R$ is   a closed subset of\/ $\Re$ which satisfies the 4-values condition. 

Let $0\in R\subseteq \Re_{\geq 0}$ which does not have $0$ as a limit.  Then there exists an Urysohn metric space $\UR$ if and only if  $R$ is a countable subset of\/ $\Re$ which  satisfies the 4-values condition.
\end{thm}
\begin{proof}
Let $0\in R\subseteq \Re_{\geq 0}$ with $0$ as a limit. 

If there exists an Urysohn metric space $\UR$  it follows form Theorem \ref{thm:4nec} that $\UR$ satisfies the 4-values condition because Urysohn metric spaces are universal. It follows from Lemma \ref{lem:clUry0} that $R$ is a closed subset of $\Re$. 

Let $R$ be closed and satisfy the 4-values condition. It follows from Lemma \ref{lem:4valdenseR} that $R$ has a countable dense subset $T$ which satisfies the 4-values condition. Let the countable universal metric space $\boldsymbol{U}_{\hskip-2pt T}$ be given by Theorem \ref{thm:basic1}. The completion $\mathrm{M}$ of $\boldsymbol{U}_{\hskip -2pt T}$ is homogeneous and separable and complete  according to Theorem \ref{thm:join3}. It follows from Corollary \ref{cor:embedwh} that $\dist(\mathrm{M})$ is the closure of $T$ which implies because $T$ is dense in $R$ and $R$ is closed that $\dist(\mathrm{M})=R$. According to Lemma~\ref{cor:realint} it remains to prove that $\mathrm{M}$ is universal, that is that every finite metric space $\mathrm{A}\in \mathcal{F}_R$ has an isometric embedding into $\mathrm{M}$. This then indeed follows from Lemma  \ref{lem:unemb5} and Theorem \ref{thm:embedwh}

\vskip 5pt
\noindent
Let $0\in R\subseteq \Re_{\geq 0}$ which does not have $0$ as a limit. 

If $R$ is uncountable then there does not exist an Urysohn metric space $\UR$ because Urysohn metric spaces are separable. If $R$ is countable then there exists universal metric space $\UR$ according to Theorem \ref{thm:basic1}. The space $\UR$ is an Urysohn metric space because the completion of $\UR$ is equal to $\UR$. 
\end{proof}

\section{Examples}

\begin{example}\label{ex:Urysph}{\normalfont
It is not difficult to check that the set of reals  in the intervals $[0,\infty)$ and  $[0,1]$ satisfy the 4-values condition. Hence there exist according to Theorem \ref{thm:haupt} an Urysohn space $\boldsymbol{U}_{\hskip -1pt [0,\infty)}$, the classical {\em Urysohn space} and  $\boldsymbol{U}_{\hskip -1pt[0,1]}$, the {\em Urysohn sphere}. 

}\end{example}

\begin{example}\label{ex:notclosed}{\normalfont The set of distances of the completion of a metric space need not be closed:

Let $R$ be the set of rationals in the interval $[0,1]$ and $V=\{a_i\mid i\in R\}\cup \{b_i\mid i\in R\}$. Let $\mathrm{V}=(V;\de)$ be the metric space with $\de(a_i,b_i)=i$ and $\de(a_i,a_j)=\de(a_i,b_j)=1$ for all $i,j\in R$ with $i\not=j$. The completion of $\mathrm{V}$ is $\mathrm{V}$.}
\end{example}

\begin{example}\label{ex:compluniv}{\normalfont The completion of a universal metric space $\UR$  need not be universal but is homogeneous according to Theorem \ref{thm:join3}. The age of the completion consists of all finite metric spaces which can be ``approximated'' by metric spaces with distances in $R$, Theorem \ref{thm:embedwh} and Lemma \ref{lem:unemb5}.

Let $R$ be the set of rationals in the interval $[0,1)$ together with the number 2. Then $R$ satisfies the 4-values condition because: Let $x\leadsto (a,b,c,d)$ with $x,a,b,c,d\in R$. According to Lemma \ref{lem:verify4} it suffices to assume $a>\max\{b+c,b+d,c+d\}$. If $a\in [0,1)$ then $b+c\leadsto (a,d,c,b)$. If $a=2$ then $b=2$ or $x=2$. If $b=2$ then $2\leadsto (a,d,c,b)$. If $x=2$ then $c=2$ or $d=2$. If $c=2$ then $2\leadsto (a,d,c,b)$ and if $d=2$ then $b+c\leadsto (a,d,c,b)$. 

Hence, according to Theorem \ref{thm:basic1}, there exists a universal countable metric space $\UR$. Let $\mathrm{M}$ be the completion of $\UR$. According to Corollary \ref{cor:embedwh} $\dist(\mathrm{M})=[0,1]\cup \{2\}:=T$, which does not satisfy the 4-values condition because $1\leadsto(2,1,\frac{1}{2},\frac{1}{2})$ but there is no number $y\in T$ with $y\leadsto (2,\frac{1}{2},\frac{1}{2},1)$.  The class $\mathcal{F}_T$ contains a triangle with distance set $\{2,1,1\}$. The class $\mathcal{R}$ does not contain any triangle with distance set of the form $\{a,b,c\}$ with $2-\frac{1}{4}<2<2+\frac{1}{4}$ and $1-\frac{1}{4}<b,c<1+\frac{1}{4}$. Hence $\mathrm{M}$ does not contain a triangle with distance se $\{2,1,1\}$ and is therefore not universal. Theorem \ref{thm:embedwh} characterizes the finite metric spaces in the age of $\mathrm{M}$. 

Indeed, it is not difficult to check that $\mathrm{M}$ consists of countably many copies of the Urysohn space $\boldsymbol{U}_{\hskip -2pt[0,1]}$ for which any two points in different copies have distance 2. 
}\end{example}

\begin{example}\label{ex:compUry}{\normalfont  Some additional examples of subsets $R$ of the reals satisfying the 4-values condition. For examples of finite sets $R$ see \cite{Ng2}.

Let $R$ be the union of the closed intervals 
\[
R=\{0\}\cup \bigcup_{n\in \omega}\left[\frac{1}{2^{2n+1}}, \frac{2^8-1}{2^{2n+4}}\right].
\]
Then $R$ is closed has $0$ as a limit and it is not difficult to check that it satisfies the 4-values condition. Hence there exists an Urysohn metric space $\UR$.

Every sum closed set $R\subseteq \Re$ containing 0 satisfies the 4-values condition.  The intersection of a sum closed set $R\subseteq \Re$ containing 0 with an interval of the form $[0,a]$ or $[0,a)$ satisfies the 4-values condition. 

The set $\omega$ and every initial interval of $\omega$ satisfies the 4-values condition. 

The sets $R=[0,1]\cup [3,4]\cup [9,\infty)$ and $R=[0,1]\cup [3,4]\cup (8,\infty)$ satisfy the 4-values condition. The set $R=[0,1]\cup [3,4]\cup [8,\infty)$ does not satisfy the 4-values condition. 

It is a bit more challenging to prove that the set of  Cantor numbers $C$ with finitely many digits 2 in the ternary expansion satisfy the 4-values condition, manuscript. Hence there exists a unique countable universal metric space $\boldsymbol{U}_{\hskip -2pt C}$ according to Theorem \ref{thm:basic1}.   The set of all  Cantor numbers does not satisfy the 4-values condition but it follows from Theorem \ref{thm:afterth} that there exists separable complete homogeneous metric space $\boldsymbol{U}_{\hskip -2pt \bar{C}}$  whose set of distances is the set of Cantor numbers. It follows from Theorem \ref{thm:Uryuni} that $\boldsymbol{U}_{\hskip -2pt \bar{C}}$ is the unique separable complete metric space whose age is the set of finite metric spaces which can be approximated by finite metric spaces with distances in $C$ as described in Theorem~\ref{thm:embedwh}.

}\end{example}

 \end{document}